\documentclass[12pt,reqno]{amsart}
\usepackage{mathrsfs}
\usepackage[breaklinks]{hyperref}
\usepackage[headheight=110pt,top=1.3in, bottom=1.2in, left=1in, right=1in]{geometry}

\usepackage{url}
\usepackage{amssymb}
\usepackage{amsmath}
\allowdisplaybreaks

\renewcommand{\Re}{\operatorname{Re}}

\newcommand{\g}{\gamma}

\newcommand{\z}{\zeta}

\renewcommand{\(}{\left(}
\renewcommand{\)}{\right)}
\renewcommand{\[}{\left\[}
\renewcommand{\]}{\right\]}
\newcommand{\Z}{\mathbb{Z}}
\newcommand{\K}{\mathbb{K}}

\newcommand{\C}{\mathbb{C}}

\renewcommand{\ge}{\geqslant}

\numberwithin{equation}{section}
 \theoremstyle{plain}
\newtheorem{theorem}{Theorem}[section]
\newtheorem{lemma}[theorem]{Lemma}
\newtheorem{corollary}[theorem]{Corollary}
\newtheorem{proposition}[theorem]{Proposition}

   \makeatletter
\def\proof{\@ifnextchar[{\@oproof}{\@nproof}}
\def\@oproof[#1][#2]{\trivlist\item[\hskip\labelsep\textit{#2 Proof of\
#1.}~]\ignorespaces}
\def\@nproof{\trivlist\item[\hskip\labelsep\textit{Proof.}~]\ignorespaces}

\makeatother

\makeatletter
\def\@tocline#1#2#3#4#5#6#7{\relax
  \ifnum #1>\c@tocdepth 
  \else
    \par \addpenalty\@secpenalty\addvspace{#2}%
    \begingroup \hyphenpenalty\@M
    \@ifempty{#4}{%
      \@tempdima\csname r@tocindent\number#1\endcsname\relax
    }{%
      \@tempdima#4\relax
    }%
    \parindent\z@ \leftskip#3\relax \advance\leftskip\@tempdima\relax
    \rightskip\@pnumwidth plus4em \parfillskip-\@pnumwidth
    #5\leavevmode\hskip-\@tempdima
      \ifcase #1
       \or\or \hskip 1em \or \hskip 2em \else \hskip 3em \fi%
      #6\nobreak\relax
    \dotfill\hbox to\@pnumwidth{\@tocpagenum{#7}}\par
    \nobreak
    \endgroup
  \fi}
\makeatother

\usepackage{bigints}
\usepackage{suffix}
\usepackage{mathtools}
\DeclarePairedDelimiterX\MeijerM[3]{\lparen}{\rparen}%
{\begin{smallmatrix}#1 \\ #2\end{smallmatrix}\delimsize\vert\,#3}

\newcommand\MeijerG[8][]{%
  G^{\,#2,#3}_{#4,#5}\MeijerM[#1]{#6}{#7}{#8}}

\WithSuffix\newcommand\MeijerG*[7]{%
  G^{\,#1,#2}_{#3,#4}\MeijerM*{#5}{#6}{#7}}
  
  \usepackage{color}
\usepackage{amsmath}

\definecolor{blue}{rgb}{0,0,1}
\definecolor{red}{rgb}{1,0,0}
\definecolor{green}{rgb}{0,.6,.2}
\definecolor{purple}{rgb}{1,0,1}

\numberwithin{theorem}{section}
\numberwithin{equation}{section}

\begin{document}
\title[Laurent coefficients of Zagier-type zeta function]{Laurent coefficients of Zagier-type zeta function {\`a} la Ishibashi and 
arithmetic aspects of extended Ramanujan period function}
\author{Soumyarup Banerjee and Riya Mandal}
\address{Department of Mathematics, Indian Institute of Technology Kharagpur, Kharagpur, Midnapore - 721302, West Bengal, India.}
\email{soumyarup@maths.iitkgp.ac.in}

\address{Department of Mathematics, Indian Institute of Technology Kharagpur, Kharagpur, Midnapore - 721302, West Bengal, India.}
\email{riyamandaljkc@gmail.com}

\begin{abstract}
One of the remarkable contributions of Don Zagier was the  Kronecker limit formula for a real quadratic field, where he connects the double series $\mathcal{Z}(s,w,w^\prime)$ to the Dedekind zeta function associated to a real quadratic field. Later, Ishibashi determined all the Laurent coefficients of $\mathcal{Z}(s,w,w^\prime)$ at $s=1$. Recently, Choie and kumar have studied the analytic behaviour of the analogous double series $\tilde{\mathcal{Z}}(s,w,w^\prime)$. In this article, we derive all the Laurent coefficients of $\tilde{\mathcal{Z}}(s,w,w^\prime)$, akin to Ishibashi. These Laurent coefficients involve an interesting function $\mathfrak{F}_k^0(x)$, which was earlier studied by Dixit et. al. (Ramanujan for $k=1$), where they obtained a beautiful symmetric relation for $\mathfrak{F}_k^0(x)$. We establish both the two term and the three term functional equation of $\mathfrak{F}_k^0(x)$, derive the action of the period-like Hecke operator on $\mathfrak{F}_k^0(x)$ and connect an important integral with $\mathfrak{F}_k^0(x)$.
\end{abstract}
\maketitle
\vspace{-0.8cm}

\section{Introduction}\label{intro}
The Laurent series of a Dedekind zeta function associated to any number field is a fundamental object of study in number theory. It finds applications not only in different branches of mathematics, such as geometry \cite{McIntyre}, but also in physics, for instance, in string theory, where it appears in one-loop computations within Polyakov’s perturbative framework (cf. \cite{Dhoker} and references therein). For any number field $\K$, if $A$ denotes an ideal class from the ideal class group of $\K$, then the Dedekind zeta function $\zeta_\K(s)$ can be decomposed as $\zeta_\K(s)= \sum_A \zeta(s, A)$, where for $\Re(s)>1$, $\zeta(s, A) = \sum_{\mathrm{a}\in A} \frac{1}{\mathcal{N}(\mathrm{a})^s}$, with $\mathcal{N}(\mathrm{a})$ being the norm of the ideal $\mathrm{a}$. The constant term in the Laurent series of $\zeta(s, A)$ at the pole $s=1$ can be described by the Kronecker limit formula, which was named after Kronecker \cite{Kronecker} for his contribution in the case of an imaginary quadratic field. Hecke \cite{Hecke} derived the corresponding formula in the case of a real quadratic field, however, it was not fully explicit, as one of the terms was left in the form of an integral involving the Dedekind eta-function. There are numerous Kronecker-type limit formulas available in the literature (cf. \cite{Cogdell}, \cite{Jorgenson}, \cite{Liu}, \cite{Ramachandra}, \cite{Vlasenko}).

In a seminal paper \cite{Zagier}, Zagier derived an explicit formulation of the Kronecker limit formula for real quadratic fields.
To state the result, we need to fix some notations. Let $\K$ be a real quadratic field with discriminant $D>0$. We call a number $w\in \K $ reduced if its conjugate $w'$ satisfies the relation $w >1>w'>0$. It is well-known that a number in $\K$ is reduced if and only if it's associated continued fraction is pure periodic. We refer \cite[p. 163]{Zagier} for more details. Now, if $\mathscr{B}$ is a narrow ideal class of $\K$ with length $r$ associated to the cycle $((b_1,...,b_r))$ with $b_i\in \Z, b_i\ge 2$, then there are exactly $r$ many reduced numbers $w_k$ with $1\le k \le r$, which has pure periodic continued fraction associated to the cycle $((b_1,...,b_r))$, for which $\{1,w_k\}$ is a basis for some ideal in $\mathscr{B}$. Zagier \cite{Zagier} first expressed the partial zeta function $\zeta(s,\mathscr{B})$ as
\begin{align*}
 		\zeta(s,\mathscr{B})= D^{-\frac{s}{2}}\sum_{k=1}^r \mathcal{Z}(s,w_k,w_k^\prime),
 \end{align*}
where the function
\begin{align*}
	\mathcal{Z}(s,w,w^\prime):=\sum_{p=1}^{\infty}\sum_{q=0}^{\infty} \bigg[\frac{(w-w^\prime)}{(q+pw)(q+pw^\prime)}\bigg]^s \quad \qquad (\Re(s)>1),
	\end{align*}
can be termed as Zagier's zeta function. This function can be continued analytically in the plane $\Re(s)>\frac{1}{2}$, with a simple pole at $s=1$. He evaluated the constant term in the Laurent series of $\mathcal{Z}(s,w,w^\prime)$ at $s=1$ to derive the Kronecker limit formula, which precisely states as 
\begin{align}\label{Zagier result}
 		\mathcal{Z}(s,w,w^\prime) =\frac{\frac{1}{2}\log(w/w')}{s-1}+ P(w, w') +\mathcal{O}(s-1),
 \end{align} 
where the function $P(x, y)$ for $x>y>0$ is defined as
\begin{align*}
P(x, y):= F(x) - F(y) + \operatorname{Li}_2\left(\frac{y}{x}\right)- \frac{\pi^2}{6} + \log \left(\frac{x}{y}\right)\left(\gamma - \frac{1}{2}\log(x-y) + \frac{1}{4}\log\!\left(\frac{x}{y}\right)\right).
\end{align*}
Here $\gamma$ denotes the Euler's constant, $\mathrm{Li}_s(z)$ is the Polylogarithm function
\begin{equation}\label{Polylog series}
\mathrm{Li}_s(z):=\sum_{n=0}^{\infty}\frac{z^n}{n^s},
\end{equation}
defined for any complex numbers $s, z$ with $|z| < 1$ and
\begin{align}\label{Herglotz}
F(x):= \sum_{n=1}^{\infty} \frac{\psi(nx) - \log(nx)}{n} \quad \qquad (x \in \C\setminus (-\infty, 0]), 
\end{align}
where $\psi(s) = \frac{\Gamma'(s)}{\Gamma(s)}$ is the logarithmic derivative of the gamma function. The convergence of the series in \eqref{Herglotz} follows from the asymptotic expansion of $\psi(s)$ given by \cite[p. 86, Equation (6)]{Maier}
\begin{align*}
\psi(s) = \log(s)+ \mathcal{O}\left(\frac{1}{s}\right) \qquad \quad (|\arg(s)|< \pi)
\end{align*}
as $s\to \infty$. The function similar to $F(x)$ was independently studied by Herglotz \cite{Herglotz}, which leads Readchenko and Zagier \cite{Radchenko} to term the function as the {\it Herglotz function}.

Egami \cite{Egami} simplifies Zagier's proof of Kronecker limit formula given in  \eqref{Zagier result}, by introducing a nice decomposition of the function $\mathcal{Z}(s,w,w^\prime)$. Later, Ishibashi \cite[Theorem 3]{Ish03} applied Egami's decomposition to determine all the Laurent coefficients of $\mathcal{Z}(s,w,w^\prime)$ at $s=1$.
The representations of the Laurent coefficients of $\mathcal{Z}(s,w,w^\prime)$ contains the integral
 	\begin{align}\label{int I}
 		I_{a,b,c}(u,v)=\int_{v}^{u}\frac{\log^ax\log^b(u-x)\log^c(x-v)}{x} dx,
 	\end{align}
defined for $u>v>0$, $a,b,c \in \mathbb{N}_0$ and the $k$-th order Herglotz function
 \begin{align*}
 	\Phi_k(x)=\sum_{n=1}^{\infty} \frac{k\psi_{k-1}(nx)-\log^k{(nx)}}{n} \quad \qquad (x \in \mathbb{C}\setminus (-\infty,0]),
 	 \end{align*} 
 	 where $\psi_k(x) = \frac{\Gamma_k'(x)}{\Gamma_k(x)} $ is the logarithmic derivative of the generalized gamma function $\Gamma_k(x)$, given by
 	 \begin{align*}
 	 	\Gamma_k(x):= \lim_{n \to \infty} \frac{\exp{(\frac{\log^{k+1}(n)}{k+1}x)}\prod_{j=1}^{n}\exp{(\frac{\log^{k+1}(j)}{k+1})}}{\prod_{j=0}^{n}\exp{(\frac{\log^{k+1}(j+z)}{k+1})}},
 	 \end{align*}
which was introduced by Dilcher in \cite{Dilcher}. The function $\psi_k(x)$ has significant importance in number theory as it is directly connected to the generalized Stieltjes constant $\gamma_k(x)$ via the relation $\psi_k(x) = - \gamma_k(x)$. In a beautiful article \cite{Dixit}, Dixit et. al. obtained the two term functional equation of the function $\Phi_k(x)$ and the three term functional equation for its derivative.
 
Recently, Choie and Kumar \cite{Choie} investigated a Kronecker limit formula analogous to the Zagier’s  result \eqref{Zagier result}. They considered the function 
	\begin{align*}
		\tilde{\mathcal{Z}}(s,w,w^\prime):=\sum_{p=1}^{\infty}\sum_{q=0}^{\infty} \bigg[\frac{p(w-w^\prime)}{(q+pw)(q+pw^\prime)}\bigg]^s,
	\end{align*}
which is defined in the region $\Re(s)>2$ and has an analytic continuation in the plane $\Re(s)>\frac{1}{2}$ with  two simple poles at $s=1$ and $s=2$. They evaluated the coefficient of the principal part and the constant term in the Laurent expansion of $\tilde{\mathcal{Z}}(s,w,w^\prime)$ at $s=1$, which precisely states as  
	\begin{align}\label{Kronecker of ztilde}
		\tilde{\mathcal{Z}}(s,w,w^\prime)= \frac{\frac{w-w^\prime}{2ww^\prime}}{s-1}+\mathfrak{F}_1(w)-\mathfrak{F}_1(w')-\frac{1}{2}\log{\frac{w}{w'}}+&\frac{w-w^\prime}{2ww^\prime}\bigg[\gamma+\log{\frac{w-w'}{ww'}}\bigg]+\mathcal{O}(s-1),
	\end{align}
where 
\begin{align}\label{Ramanujan-F1-Function}
	\mathfrak{F}_1(x):= \sum_{n=1}^{\infty} \Bigg[\psi(nx)+\frac{1}{2nx}-\log(nx)\bigg] \quad \qquad \left(x \in \mathbb{C}\setminus(-\infty,0]\right).
\end{align}
The above function was previously studied by Ramanujan \cite{Ramanujan}, where on page 220, he recorded an elegant symmetric relation
\begin{align}\label{Ramanujan modular relation}
	\sqrt{x}\left\{\mathfrak{F}_1(x)+\frac{\gamma-\log(2\pi x)}{2x}\right\}=\sqrt{1/x} \left\{ \mathfrak{F}_1\left(\frac{1}{x}\right)+ \frac{\gamma-\log(\frac{2\pi}{x})}{2/x}\right\}.
\end{align} 
Later, Choie and Kumar \cite{Choie} explore the other arithmetic properties of the above function $\mathfrak{F}_1(x)$.

In this article, our primary goal is to determine all the Laurent coefficients of $\tilde{\mathcal{Z}}(s,w,w^\prime)$ at $s=1$ and for that we need to consider the function
\begin{align}\label{New-Fk}
\mathfrak{F}_k(x):=	\sum_{r=0}^{k-1}\binom{k-1}{r}	\sum_{n=1}^{\infty}\log^r\left(\frac{1}{\sqrt{n}}\right)\left[\psi_{k-r-1}(nx) + \frac{\log^{k-r-1}(nx)}{2nx} -\frac{\log^{k-r}(nx)}{k-r}\right],
\end{align}
which is defined for any positive integer $k$ and $x \in \mathbb{C}\setminus(-\infty,0]$. Clearly, for $k=1$, the above function reduces to \eqref{Ramanujan-F1-Function}. The convergence of the series, appeared in \eqref{New-Fk}, follows from the asymptotic expansion of $\psi_{k-r-1}(s)$ \cite[Proposition 3]{Coffey}, which states that (cf. \cite[p. 7, Equation (2.3)]{DSS24}), for $|\arg(s)|<\pi$, as $s\to \infty$,
        \begin{align*}
    \psi_{k-r-1}(s) \sim \frac{\log^{k-r}(s)}{k-r}&-\frac{1}{2s}\log^{k-r-1}(s) \\
    &+ \sum_{m=1}^{\infty} \frac{B_{2m}}{(2m)!}s^{-2m}\sum_{t=0}^{k-r-1}\binom{k-r-1}{t} t! s(2m,t+1) \log^{k-r-1-t}(s),
        \end{align*} 
where $s(n,k)$ denotes the Stirling number of the first kind and $B_m$ denotes the $m$-th Bernoulli number. Moreover, the above asymptotic expansion readily implies that for $|\arg(x)|<\pi$, as $x\to \infty$,
\begin{align*}
\mathfrak{F}_k(x)\sim \sum_{r=0}^{k-1}\binom{k-1}{r}\sum_{n=1}^{\infty}\log^r\left(\frac{1}{\sqrt{n}}\right)&\sum_{m=1}^{\infty} \frac{B_{2m}}{(2m)!}(nx)^{-2m}\\
\times&\sum_{t=0}^{k-r-1}\binom{k-r-1}{t} t! s(2m,t+1)\log^{k-r-1-t}(nx).
\end{align*}
We next define two important integrals
\begin{align}
		&\tilde{H}(u,v):=-\int_{0}^{\infty}t^{v-1}\bigg(\frac{1}{e^t-1}-\frac{1}{t}+\frac{1}{2}\bigg)\mathrm{Li}_{-\frac{(v-1)}{2}}(e^{-ut}) dt \qquad \quad  (u>0, v>-1), \label{Hfun}\\
&J_{a,b,c}(u,v):=\int_{v}^{u}\frac{\log^ax\log^b(u-x)\log^c(x-v)}{x^2} dx\qquad \qquad \quad  (u>v>0), 	\label{int J}
\end{align}
where $a, b, c$ are any non-negative integers. We are now ready to state our first result, where we evaluate  all the Laurent coefficients of $\tilde{\mathcal{Z}}(s,w,w^\prime)$ at $s=1$.
\begin{theorem}\label{Thm-1.2}
The Laurent series expansion of $\tilde{\mathcal{Z}}(s,w,w^\prime)$ at $s=1$ is given by
	\begin{align*}
		\tilde{\mathcal{Z}}(s,w,w^\prime)=\sum_{k=-1}^{\infty}\tilde{P}_k(w,w^\prime)(s-1)^k \quad\qquad \text{with} \quad \tilde{P}_k(w,w^\prime)=\sum_{i+j=k}a_i(b_j+c_j+d_j),
	\end{align*}
	where 
\begin{align}
		a_i&= \sum_{l+m=i}\frac{(-1)^l}{l!} \sum_{\sum_{r=1}^mrs_r=m} \bigg[\prod_{j=1}^m (s_j+1)\bigg(\frac{-\Gamma^{(j)}(1)}{j!}\bigg)^{s_j}\bigg]\log^l(w-w^\prime) \hspace{1.8cm} (i\ge 0),\nonumber\\
		b_j&=\frac{2^j}{j!}\sum_{l=0}^j b_{j,l+1}\bigg[	\mathfrak{F}_{l+1}(w)-	\mathfrak{F}_{l+1}(w')\bigg]\nonumber\\
		&\hspace{2.1cm}+\frac{1}{j!}\sum_{r=0}^{j-1} \binom{j}{r} 2^r \int_{w^\prime}^w \log^{j-r}\{(w-u)(u-w^\prime)\} \frac{ \partial}{\partial u}[\partial^r_s\tilde{H}(u,1)] du \quad (j\ge 0),\nonumber\\
		c_j&=\sum_{l_1+l_2+l_3=j}\sum_{q+r\le l_3} \frac{(-1)^r 2^{l_2+r}\z^{(l_1)}(0)\Gamma^{(l_2)}(1)}{l_1!l_2!l_3!} \binom{l_3}{r}\binom{l_3-r}{q} I_{r,q,l_3-r-q}(w,w^\prime) \quad (j\ge 0),\nonumber\\
		d_j&=\sum_{\substack{l_1+l_2+l_3=j \\ l_1\ge -1}}\sum_{q+r\le l_3}\frac{(-1)^r(2)^{r+l_2} \tilde{\gamma}_{l_1}\Gamma^{(l_2)}(2)}{l_2! l_3!}\binom{l_3}{r} \binom{l_3-r}{q} J_{r,q,l_3-r-q}(w,w^\prime) \quad (j\ge -1).\nonumber
	\end{align}
Here $b_{j,l+1}$ are the computable constants, defined in Section \ref{Laurent series} at \eqref{bkj} and $\tilde{\gamma}_{m}$ denotes the $m$-th Laurent coefficient of $\zeta(s)$ at $s=1$.
\end{theorem}
Note that, $\tilde{P}_{-1}(w,w')$ and $\tilde{P}_0(w,w')$ from the above theorem denotes the coefficient of the principal part and the constant term respectively in the Laurent expansion of $\tilde{\mathcal{Z}}(s,w,w^\prime)$ at $s=1$. Therefore, an immediate corollary of the above theorem is the result of Choie and Kumar, given in \eqref{Kronecker of ztilde}.
\begin{corollary}\label{cor Rahul} 
The following identities hold:
		\begin{align*}
			&\tilde{P}_{-1}(w,w')=\frac{w-w^\prime}{2ww^\prime}, \\ \text{ and } \qquad
			& 
			\tilde{P}_0(w,w')=\mathfrak{F}_1(w)-\mathfrak{F}_1(w')-\frac{1}{2}\log{\frac{w}{w'}}+\frac{w-w^\prime}{2ww^\prime}\bigg[\gamma+\log{\frac{w-w'}{ww'}}\bigg].
		\end{align*}
\end{corollary}
Choie and Kumar in \cite{Choie} investigated the arithmetic properties and application of $\mathfrak{F}_1(x)$ and for that they mainly considered the function 
\begin{align*}
    \mathcal{F}_1(x):=\mathfrak{F}_1(x)-\frac{1}{2}\left(\gamma-\log\left(\frac{2\pi}{x}\right)\right) \quad \qquad (x \in \mathbb{C}\setminus(-\infty,0]).
\end{align*}
Ramanujan's identity \eqref{Ramanujan modular relation} leads to the two term functional equation
\begin{equation}\label{Two term F1}
\mathcal{F}_1(x)=\frac{1}{x}\mathcal{F}_1\left(\frac{1}{x}\right).
\end{equation} 
Moreover, the above function also satisfies period relations of Maass cusp forms in the sense of Lewis-zagier \cite{Lewis}, that is it satisfies three term functional equation
\begin{align}\label{Three term F1}
    \mathcal{F}_1(x)-\mathcal{F}_1(x+1)-\frac{1}{x+1}\mathcal{F}_1\left(\frac{x}{x+1}\right)=0,
\end{align}
which led Choie and Kumar to aptly coin, for the function $\mathcal{F}_1(x)$, the terminology {\it Ramanujan period function}. In the same article, the authors have shown that the function $\mathcal{F}_1(x)$ is a Hecke eigenform with respect to the Hecke operators acting on the space of periods (cf. \cite{Choiezag}, \cite{Radchenko}) and also provide important applications of the Ramanujan period function.

We next draw our attention in the first term ($r=0$-th term) of the finite sum in $\mathfrak{F}_k(x)$, which we denote by $\mathfrak{F}_k^0(x)$. For any positive integer $k$, the definition \eqref{New-Fk} of $\mathfrak{F}_k(x)$ yields
\begin{align}\label{Dixit's-Fk}
	\mathfrak{F}_k^0(x)= \sum_{n=1}^{\infty} \bigg[ \psi_{k-1}(nx)+\frac{\log^{k-1}(nx)}{2nx}-\frac{\log^{k}(nx)}{k}  \bigg] \quad \qquad \left(x \in \mathbb{C}\setminus(-\infty,0]\right).
\end{align}
Recently, Dixit et. al.	\cite[Theorem 2.1]{DSS24} established a beautiful symmetric relation involving $\mathfrak{F}_k^0(x)$, which precisely states that
\begin{align*}
	&\sqrt{x}\sum_{j=0}^k (-1)^{j+1}\binom{k}{j} \log^{k-j}(\sqrt{x}) \Bigg\{ \mathfrak{F}_{j+1}^0(x)
	+\sum_{l=0}^j \binom{j}{l} \frac{\gamma_\ell \log^{j-l}(x)}{2x} - \frac{\log^{j+1}(x)+2(-1)^{j+1}\zeta^{j+1}(0)}{2(j+1)x} \Bigg\}\nonumber\\
	& \qquad =\sqrt{\frac{1}{x}}\sum_{j=0}^k (-1)^{j+1}\binom{k}{j} \log^{k-j}\left(\sqrt{\frac{1}{x}}\right)\hspace{-.15cm} \Bigg\{ \mathfrak{F}_{j+1}^0\left(\frac{1}{x}\right)
	+\sum_{l=0}^j \binom{j}{l} \frac{\gamma_\ell \log^{j-l}(\frac{1}{x})}{2/x}\\
	&\hspace{10cm}- \frac{\log^{j+1}(\frac{1}{x})+2(-1)^{j+1}\zeta^{j+1}(0)}{2(j+1)/x} \Bigg\},
\end{align*}
where $\gamma_\ell$ is the generalized Euler constant, defined in \cite[p. 164, Equation (13.1)]{Berndt}. Clearly, the above identity reduces to Ramanujan's identity \eqref{Ramanujan modular relation} at $k=1$. Here our goal is to investigate the arithmetic properties and applications of $\mathfrak{F}_k^0(x)$ akin to Choie and Kumar and for that we consider the $k$-th order extension of the Ramanujan period function given by
\begin{align}\label{Generalized-Ramanujan-Period-Function}
         \mathcal{F}_k(x):= (-1)^{k-1}\mathfrak{F}_k^0(x)-\sum_{l=0}^{k}\frac{(-1)^l}{k}\binom{k}{l} \log^{l}(x)\zeta^{(k-l)}(0)-\frac{1}{2}(-1)^{k-1}\gamma_{k-1},
\end{align}
which is defined for $x \in \mathbb{C}\setminus (-\infty, 0]$.
Our next result on the function $\mathcal{F}_k(x)$ generalizes the two term \eqref{Two term F1} and the three term \eqref{Three term F1} functional equation satisfied by the Ramanujan period function $\mathcal{F}_1(x)$.
\begin{theorem}\label{Periodlike}
    For $x \in \mathbb{C}\setminus[0,\infty)$, the function $\mathcal{F}_k(x)$ satisfies
    \begin{align}\label{two term}
         \mathcal{F}_k(x)= \frac{1}{x} \sum_{j=1}^k \binom{k-1}{j-1} \log^{k-j}\(\frac{1}{x}\) \mathcal{F}_{j}\(\frac{1}{x}\),
    \end{align}
    \begin{align}\label{three term}
       \mathcal{F}_k(x)- \mathcal{F}_k(x+1) - \frac{1}{x+1} \sum_{j=1}^k \binom{k-1}{j-1} \log^{k-j}(x+1) \mathcal{F}_j\(\frac{x}{x+1}\)=0.
    \end{align}
\end{theorem}
We next study an action of the Hecke operators on the function $\mathcal{F}_k(x)$ and for that we follow the notations of Radchenko and Zagier \cite{Radchenko}, to define the Hecke operators, which acts on the space of periods. 

Let $\Gamma$ be the full modular group, which is generated by the matrices $S=\begin{pmatrix}
	0 & -1\\
	1 & 0
\end{pmatrix}$ and $T=\begin{pmatrix}
1 & 1\\
0 & 1
\end{pmatrix}$. For any positive integer $n$, let $\mathcal{M}_n$ be the set of $2\times 2$ integer matrices of determinant $n$, modulo $\{\pm 1\}$ and $\mathcal{R}_n=\mathbb{Q}[\mathcal{M}_n]$. For every positive integer $n$, we call an element $\tilde{T}_n\in \mathcal{R}_n $ ``acts like the $n$-th Hecke operator on periods", if it satisfies the relation (cf. \cite{Choiezag}, \cite{Radchenko})
\begin{align*}
	(1-S)\tilde{T}_n=T_n^\infty (1-S)+(1-T)Y
\end{align*}
for some $Y \in \mathcal{R}_n$, where 
\begin{align*}
	T_n^\infty = \sum_{ad=n}\sum_{0\le b<d} \begin{bmatrix}
		a & b\\
		0 & d
	\end{bmatrix} \in \mathcal{R}_n,
\end{align*}
represents the usual $n$-th Hecke operator, which acts on the space of modular forms. 

For simplicity, we denote the $j$-th derivative of a general divisor function $\sigma_{1-s}(n) = \sum_{d|n}d^{1-s}$ with respect to $s$, evaluated at $s=1$ by $\sigma_{0}^{(j)}(n)$. For any non-negative integer $l$ and $\gamma\in \mathcal{M}_n$, we define the general slash operator of weight $k$ by
\begin{align}\label{genslash}
	(f|^{l}_{k}\gamma)(x):= \frac{(ad-bc)^{\frac{k}{2}}}{(cx+d)^k} \log^{l}\left(\frac{1}{cx+d}\right) f\left(\frac{ax+b}{cx+d}\right).
\end{align}
It reduces to the standard slash operator at $l=0$. We are now ready to state our result, where we exhibit the action of $\tilde{T}_n$ on the higher order extended Ramnujan period function, defined in \eqref{Generalized-Ramanujan-Period-Function}. 
\begin{theorem}\label{Heckeopaction}
    Let $\tilde{T_n}=\sum_{\gamma}v_{\gamma}\gamma$ represent the $n$-th Hecke operator on periods and that all matrices in $\tilde{T_n}$ have non-negative entries. For $x>0$, the higher order extended Ramanujan period function satisfies
    \begin{align*}
        \sum_{j=1}^k \binom{k-1}{j-1} \bigg(\mathcal{F}_j|_1^{(k-j)}\tilde{T_n}\bigg)(x) = \sqrt{n} \sum_{j=1}^{k} \binom{k-1}{j-1}\sigma_{0}^{(k-j)}(n) \mathcal{F}_j(x).
    \end{align*}
\end{theorem}
We next consider an example \cite[Proposition 3]{Radchenko} 
\begin{align*}
    \hat{T}_n=\sum_{\substack{0\le c <a \\ 0\le b <d \\ ad-bc=n}}\begin{bmatrix}
        a &b\\
c & d    \end{bmatrix},
\end{align*}
which acts like the $n$-th Hecke operator on periods. In particular, for $n=2$ and $3$, we have
\begin{align}\label{Example}
	\hat{T}_2= \begin{bmatrix}
		1 & 0\\
		0 & 2
	\end{bmatrix}+ \begin{bmatrix}
	1 & 1\\
	0 & 2
	\end{bmatrix}+ \begin{bmatrix}
	2 & 0\\
	0 & 1
	\end{bmatrix}+ \begin{bmatrix}
	2 & 0\\
	1 & 1
	\end{bmatrix}, \, \, \, 
	\hat{T}_3=\sum_{p=0}^{2}\begin{bmatrix}
		1 & p\\
		0 & 3
	\end{bmatrix} +\sum_{p=0}^{2}\begin{bmatrix}
	3 & 0\\
	p & 1
	\end{bmatrix} +\begin{bmatrix}
	2 & 1\\
	1 & 2
	\end{bmatrix}.
\end{align}
The following corollary shows that one can construct multi-term functional equation for the higher order extended Ramanujan period function.
\begin{corollary}
For $x>0$, higher order extended Ramanujan period function satisfies
 \begin{align*}
 \mathcal{F}_k(2x)&= \sum_{j=1}^{k}\binom{k-1}{j-1}\bigg[\sigma_0^{(k-j)}(2)\mathcal{F}_j(x)-\frac{1}{2}\log^{k-j}\left(\frac{1}{2}\right)\mathcal{F}_j\left(\frac{x}{2}\right)-\frac{1}{2}\log^{k-j}\left(\frac{1}{2}\right)\mathcal{F}_j\left(\frac{x+1}{2}\right)\\
 &\hspace{8.1cm} -\frac{1}{x+1}\log^{k-j}\left(\frac{1}{x+1}\right)\mathcal{F}_j\left(\frac{2x}{x+1}\right) \bigg],\\
\mathcal{F}_k(3x)&=\sum_{j=1}^{k}\binom{k-1}{j-1}\bigg[\sigma_0^{(k-j)}(3)\mathcal{F}_j(x)-\frac{1}{3}\log^{k-j}\left(\frac{1}{3}\right)\sum_{p=0}^{2}\mathcal{F}_j\left(\frac{x+p}{3}\right)\\
&\hspace{0.2cm}-\frac{1}{x+2}\log^{k-j}\left(\frac{1}{x+2}\right)\mathcal{F}_j\left(\frac{2x+1}{x+2}\right)-\sum_{p=1}^{2}\frac{1}{px+1}\log^{k-j}\left(\frac{1}{px+1}\right)\mathcal{F}_j\left(\frac{3x}{px+1}\right)\bigg].
    \end{align*}
\end{corollary}
The above result readily follows from an application of Theorem \ref{Heckeopaction}, where we need to substitute $\tilde{T_n}$ by $\hat{T}_2$ and $\hat{T}_3$ respectively, as given in \eqref{Example}, to act on the higher order extended Ramanujan period function.  

Over the years, the integral \cite[p. 145, 271]{Cohen}
\begin{align*}
J(x)= \int_{0}^{\infty} \frac{\log\!\left(1 + e^{-xt}\right)}{1 + e^{t}}~ dt \quad \qquad (\Re(x) > 0),
\end{align*}
has been the subject of sustained research interest. Regarding the above integral, Chowla \cite[p. 372]{Chowla} remarked, `A direct evaluation of this definite integral is probably difficult'. Recently, Radchenko and Zagier \cite{Radchenko} established a beautiful connection between $J(x)$ and the Herglotz function $F(x)$, namely, for $x>0$
\begin{align*}
	J(x)=F(2x)-2F(x)+F(\frac{x}{2})+\frac{\pi^2}{12x},
\end{align*}
which leads them to evaluate special values of $J(x)$, for example,
\begin{align*}
J(4+\sqrt{17}) &= -\frac{\pi^2}{6}+\frac{1}{2}\log^2(2)+ \frac{1}{2}\log(2)\log(4+\sqrt{17}),\\
J\left(\frac{2}{5}\right) &= \frac{11\pi^2}{240}+\frac{3}{4}\log^2(2)-2\log^2\left(\frac{\sqrt{5}+1}{2} \right).
\end{align*}
Earlier works of Herglotz \cite{Herglotz} as well as Muzaffar and Williams \cite{Muzaffar} on evaluating the special values of $J(x)$, in the particular form $J(n+\sqrt{n^2-1})$, are based on the tools of algebraic number theory, whereas, Radchenko and Zagier \cite{Radchenko} employed the methods from analytic number theory to evaluate the special values of $J(x)$, which makes this integrals more elegant. Recently, Dixit et. el. \cite{Dixit} studied more general integral involving Gauss sum.

Choie and Kumar \cite{Choie} introduced an integral, which is an analogue of $J(x)$, namely
\begin{align*}
	J_1(x):= \int_{0}^{\infty} \frac{dt}{(1+e^{-t})(1+e^{xt})} \quad \qquad (\Re(x)>0),
\end{align*}
and established a nice relation with the function $\mathfrak{F}_1(x)$, defined in \eqref{Ramanujan-F1-Function}. The relation precisely states as \cite[Theorem 4.5]{Choie}
\begin{align}\label{CR-J-relation-F}
	J_1(x)=3\mathfrak{F}_1(x)-2\mathfrak{F}_1(2x)-\mathfrak{F}_1(\frac{x}{2})+\frac{1}{2x}\log{(2)}.
\end{align}
Here, for any posive integer $k$, we consider the k-th order generalization of the above integral namely,
\begin{align}\label{J_k-Function}
    J_k(x) = \int_{0}^{\infty} \frac{S_k(t)}{(1+e^{-t})(1+e^{xt})} dt \quad \qquad  (\Re(x)>0),
\end{align}
where $S_k(t)$ is the $k$-th Ishibashi polynomial, defined in \eqref{Ipoly}. In the  following theorem, we generalize the identity \eqref{CR-J-relation-F} by connecting the above integral with the function $\mathfrak{F}_k^0(x)$, defined in \eqref{Dixit's-Fk}. 
   \begin{theorem}\label{JrelF}
    For $\Re(x)>0$, we have
    \begin{align*}
J_k(x)=\mathfrak{F}_k^0(x)&-2\mathfrak{F}_k^0(2x)-\sum_{j=1}^k \binom{k-1}{j-1}\log^{k-j}\left(\frac{1}{2}\right)\bigg[\mathfrak{F}_j^0\left(\frac{x}{2}\right)-2\mathfrak{F}_j^0(x)\bigg]\nonumber\\
    &+\frac{(-1)^k}{2kx}\bigg[\log^k x-\log^k2x\bigg]-\frac{(-1)^{k}}{2x}\sum_{l=0}^{k-1}\binom{k-1}{l}\gamma_{k-1-l}\bigg[\log^lx-\log^l2x\bigg].
\end{align*}
\end{theorem}

The paper is organized as follows. In Section \ref{Prelim}, we recall some useful notations an results. Section \ref{Laurent series} is concerned about the Laurent series expansion of the function $\tilde{\mathcal{Z}}(s,w,w^\prime)$ at $s=1$. In Section \ref{Arithmetic properties}, we investigate the arithmetic properties of $\mathfrak{F}_k^0(x)$, which naturally arises in the Laurent coefficient of $\tilde{\mathcal{Z}}(s,w,w^\prime)$. Finally, in Section \ref{Application}, we relate $\mathfrak{F}_k^0(x)$ with the integral $J_k(x)$.

\section{Preliminaries}\label{Prelim}
In this section, we collect known facts from the literature that will be used in the subsequent proofs. In order to evaluate the Laurent coefficients of $\mathcal{Z}(s,w,w^\prime)$ at $s=1$, Ishibashi  \cite[Theorem 1]{Ish03} studied an auxilliary polynomial in logarithm given by
\begin{align}\label{Ipoly}
 		S_k(t)=\sum_{j=0}^{k-1}a_{k,j}\log^j(t),
 	\end{align}
which can be termed as $k$-th Ishibashi polynomial, where $a_{k,j}$ are defined recursively by
 	\begin{align*}
 		a_{k,j}=-\sum_{r=0}^{k-2}\binom{k-1}{r}\Gamma^{k-r-1}(1)~a_{r+1,j} \quad \qquad (0\le j\le k-1),
 	\end{align*}
starting from $a_{1,0}=1$, and $a_{k,k-1}=1$. The following result \cite[Lemma 3.1]{Dixit} regarding the $k$-th Ishibashi polynomial \eqref{Ipoly}, is crucial throghout. 
\begin{lemma}\label{Dixit's-Lemma}
	\begin{enumerate}
		\item For $k$,$j \in \mathbb{N}  $ such that $k-j\ge1$, we have
		\begin{align*}
			a_{k,j}=\binom{k-1}{j} \lim_{s\to 1}\frac{d^{k-j-1}}{ds^{k-j-1}}\bigg[\frac{1}{\Gamma(s)}\bigg].
		\end{align*}
	
		\item 	For any positive integer $k$, we have 
		\begin{align*}
			S_k(t)= \lim_{s\to1} \frac{d^{k-1}}{ds^{k-1}}\bigg[ \frac{t^{s-1}}{\Gamma(s)}\bigg].
		\end{align*}
	\end{enumerate}

\end{lemma}

\section{Laurent series of $\mathcal{Z}(s,w,w^\prime)$ at $s=1$}\label{Laurent series}
In this section, we investigate the Laurent series expansion of the function $\mathcal{Z}(s,w,w^\prime)$ at $s=1$ and for that we decompose $\mathcal{Z}(s,w,w^\prime)$ into three parts and evaluate the Laurent coefficients of each parts individually. One of the key ingredients for these evaluation is the following function
\begin{align}\label{zetatilde}
	\tilde{\z}(s,(a_1,a_2),(x_1,x_2)):=\sum_{m=0}^{\infty}\sum_{n=0}^{\infty} \frac{(x_2+n)^s}{\prod\limits_{k=1,2}\{x_1+m+(x_2+n)a_k\}^s}, 
\end{align}
which is defined under the conditions $0<a_1<a_2$, $x\ge 0$, $x_2\ge0$, $(x_1,x_2)\neq(0,0)$ and $\Re(s)>2$. The next lemma provides the integral representations of the above function.
\begin{lemma}\label{Lem:2.1}
For $\Re(s)>2$, we have
	\begin{align*}
		\tilde{\z}(s,(a_1,a_2),(x_1,x_2))=\frac{1}{\Gamma^2(s)}\int_{0}^{\infty}\int_{0}^{\infty} (uv)^{s-1}\frac{e^{-x_1(u+v)}}{1-e^{-(u+v)}}\sum_{n=0}^{\infty}(x_2+n)^s e^{-(x_2+n)(a_1u+a_2v)} du dv.
	\end{align*}
\end{lemma}
\begin{proof}
The integral representation of $\Gamma^2(s)$	yields for $\Re(s)>1$, we have
\begin{align*}
		\Gamma^2(s)=\int_{0}^{\infty} \int_{0}^{\infty}(uv)^{s-1} e^{-u} e^{-v}du dv.
	\end{align*}
Thus by substituting $u$ by $\{x_1+m+(x_2+n)a_1\}u$ and  $v$ by $\{x_1+m+(x_2+n)a_2\}v$, the above integral reduces to 
\begin{align*}
		\frac{\Gamma^2(s)}{\prod_{k=1,2}\{x_1+m+(x_2+n)a_k\}^s}=\int_{0}^{\infty} \int_{0}^{\infty} (uv)^{s-1} e^{-x_1(u+v)} e^{-m(u+v)} e^{-(x_2+n)(a_1u+a_2v)} du dv.
	\end{align*}
Finally, we multiply both sides of the above equation by $(x_2+n)^s$ and take double sum over the variables $m$ and $n$ from $0$ to $\infty$ on both sides to obtain
	\begin{align*}
		\Gamma^2(s)\tilde{\z}(s,(a_1,a_2),(x_1,x_2))=\int_{0}^{\infty}\int_{0}^{\infty} (uv)^{s-1}\frac{e^{-x_1(u+v)}}{1-e^{-(u+v)}}\sum_{n=0}^{\infty}(x_2+n)^s e^{-(x_2+n)(a_1u+a_2v)} du~ dv,
	\end{align*}
which concludes our lemma.
\end{proof}

In the next proposition, we decompose the function $	\tilde{\z}(s,(w^{\prime},w),(0,1))$ into three integrals.
\begin{proposition}\label{Prop2.2}
	The decomposition 
	\begin{align*}
		\tilde{\z}(s,(w^{\prime},w),(0,1))=\frac{(w-w^{\prime})^{1-2s}}{\Gamma^2(s)}\bigg[I_1(s)+I_2(s)+I_3(s)\bigg]
	\end{align*}
holds for $\Re(s)>2$, where
	\begin{align*}
		&I_1(s)=  \int_{w^\prime}^{w}\{(w-u)(u-w^\prime)\}^{s-1}  \frac{\partial \tilde{H}(u,2s-1)}{\partial u} du,\\
		&	I_2(s)=\Gamma(2s-1)\z(s-1)\int_{w^\prime}^w u^{1-2s}\{(w-u)(u-w^\prime)\}^{s-1} du,\\
		&	I_3(s)=\frac{1}{2}\Gamma(2s)\z(s)\int_{w^\prime}^w u^{-2s}\{(w-u)(u-w^\prime)\}^{s-1} du.
	\end{align*}
Moreover, the function $\tilde{\z}(s,(w^{\prime},w),(0,1))$  has analytic continuation to the half plane $\Re(s)>\frac{1}{2}$ with simple poles at $s=1$ and $s=2$. 
	
\end{proposition}
\begin{proof}
The integral $I_1(s)$ converges for $\Re(s)>0$, which follows from the series representation \eqref{Polylog series} of $\mathrm{Li}_s(z)$, which occurs in the definition \eqref{Hfun} of $\tilde{H}(u,v)$ and the well-known series expansion of $\frac{t}{e^t-1}$ in terms of Bernoulli numbers (cf. \cite[p. 264]{Apostol}).
The integral $I_2(s)$ is analytic in $\Re(s)>\frac{1}{2}$ except for simple poles at $s=\frac{1}{2}$, $s=2$ and the integral $I_3(s)$ is analytic $\Re(s)>\frac{1}{2}$ except for simple poles at $s=1$. Both of the analytic properties of $I_2(s)$ and $I_3(s)$ follows directly from the analytic behaviour of Euler Gamma function and Riemann zeta function. 
	
We next apply Lemma \ref{Lem:2.1} to obtain
	\begin{align*}
	\tilde{\z}(s,(w^\prime,w),(0,1))= \frac{1}{\Gamma^2(s)}\int_{0}^{\infty}\int_{0}^{\infty} (uv)^{s-1}\frac{1}{1-e^{-(u+v)}}\mathrm{Li}_{-s}(e^{-(w^\prime u+w v)}) du dv.
	\end{align*}
Thus, the change of variables $u+v=t$ and $w^\prime u+w v=tx$  yields
	\begin{align*}
		&\tilde{\z}(s,(w^\prime,w),(0,1))\\
		&\qquad= \frac{(w-w^\prime)^{1-2s}}{\Gamma^2(s)}\int_{w^\prime}^{w}\{(w-u)(u-w^\prime)\}^{s-1}\int_{0}^{\infty} t^{2s-1}\frac{1}{1-e^{-t}}\mathrm{Li}_{-s}(e^{-ut})dt~ du \\
		&\qquad= \frac{(w-w^{\prime})^{1-2s}}{\Gamma^2(s)}\hspace{-.2cm}\int_{w^\prime}^{w}\hspace{-.2cm}\{(w-u)(u-w^\prime)\}^{s-1}\Bigg[\int_{0}^{\infty} t^{2s-1}\bigg(\frac{1}{e^t-1}-\frac{1}{t}+\frac{1}{2}\bigg) \mathrm{Li}_{-s}(e^{-ut}) dt~ du \\
		& \hspace{5cm} + \int_{0}^{\infty}t^{2s-2}\mathrm{Li}_{-s}(e^{-ut}) dt~ du + \frac{1}{2}\int_{0}^{\infty}t^{2s-1}\mathrm{Li}_{-s}(e^{-ut}) dt~ du \Bigg] \\
		&\qquad=\frac{(w-w^{\prime})^{1-2s}}{\Gamma^2(s)}\bigg[I_1(s)+I_2(s)+I_3(s)\bigg],
	\end{align*}
where in the final step, the definition \eqref{Hfun} of $\tilde{H}(u,v)$ provides the integral $I_1(s)$ and to obtain the integrals $I_2(s)$ and $I_3(s)$, we first apply the infinite series representaion of $\mathrm{Li}_{-s}(e^{-ut})$ and then interchange the summation and integration, which are justified due to the uniform convergence of the series  $\mathrm{Li}_{-s}(e^{-ut})$ on any compact interval of $(0,\infty)$. This completes the proof of the proposition.
	\end{proof}

The next lemma exhibits the integral representation of $\mathfrak{F}_k(x)$, defined in \eqref{New-Fk}.
	\begin{lemma}\label{Integral-Representation-New-Fk}
		For $\Re(x)>0$, we have
		\begin{align*}
			 \mathfrak{F}_k(x)=(-1)^{k}\sum_{r=0}^{k-1}\binom{k-1}{r}\int_{0}^\infty \bigg(\frac{1}{e^t-1}-\frac{1}{t} +\frac{1}{2}\bigg)S_{k-r}(t) \left(\sum_{n=1}^{\infty}\log^r(\sqrt{n})e^{-nxt}\right)   dt.
	\end{align*}
	\end{lemma}
	\begin{proof}
		It follows from \cite[Equations (2.8), (2.11)]{Ish03} we have  the following (cf. \cite[p. 30, Equation 9.1]{DSS24}):
		\begin{align*}
			\psi_{k-r-1}(x) + \frac{\log^{k-r-1}(x)}{2x} -\frac{\log^{k-r}(x)}{k-r}= (-1)^{k-r}\int_{0}^\infty e^{-xt} \bigg(\frac{1}{e^t-1}-\frac{1}{t} +\frac{1}{2}\bigg)S_{k-r}(t) \ dt.\nonumber
		\end{align*}
Thus, by replacing $x$ by $nx$ on the both side, we can derive $\mathfrak{F}_k(x)$ as
\begin{align*}
\mathfrak{F}_k(x)&= (-1)^{k}\sum_{r=0}^{k-1}\binom{k-1}{r}	\sum_{n=1}^{\infty}\log^r\left(\sqrt{n}\right)\int_{0}^\infty e^{-nxt} \bigg(\frac{1}{e^t-1}-\frac{1}{t} +\frac{1}{2}\bigg)S_{k-r}(t) \ dt,
\end{align*}
which concludes our lemma.
	\end{proof}
	
We abbreviate the notation of the $k$-th order partial derivative of the function $\tilde{H}(u,v)$, as defined in \eqref{Hfun}, with respect to $v$ at $v=1$ by $\partial^{(k)}_v\tilde{H}(u,1)$. The following lemma relates $\partial^{(k)}_v\tilde{H}(u,1)$ with the function $\mathfrak{F}_k(x)$.
\begin{lemma}\label{lem2.3}
For $u>0$, we have
		\begin{align*}
		 \partial^{(k)}_v\tilde{H}(u,1)
		=\sum_{j=0}^k b_{k,j+1} \mathfrak{F}_{j+1}(u),
	\end{align*}
where the constants $b_{k,j+1}$ can be defined through the recursion relation 
 \begin{align}\label{bkj}
 	b_{k,j+1}=-\sum_{r= j}^{k-1} a_{k+1,r}b_{r,j+1} \quad \qquad (0\le j\le k-1),
 \end{align}
with the initial conditions $b_{0,1}=1$ and $b_{k,k+1}=(-1)^k$.	
\end{lemma}
\begin{proof}
	We first differentiate $j$ times the function $\tilde{H}(u,v)$ with respect to $v$ to obtain
	\begin{align*}
		\frac{\partial^{j}}{\partial v^j}[\tilde{H}(u,v)]
		=-\sum_{r=0}^j \binom{j}{r} \int_{0}^{\infty} t^{v-1}\log^{j-r} t \bigg(\frac{1}{e^t-1}-\frac{1}{t}+\frac{1}{2}\bigg)\frac{\partial^r}{\partial v^r}[\mathrm{Li}_{-\frac{v-1}{2}}(e^{-ut})] dt.
	\end{align*}
Therefore, we can write
	\begin{align}\label{Akj-H}
		&\sum_{j=0}^{k}a_{k+1,j}\frac{\partial^{(j)}}{\partial v^j}[\tilde{H}(u,v)]\nonumber\\
		&=-	\sum_{j=0}^{k}a_{k+1,j}\sum_{r=0}^j \binom{j}{r} \int_{0}^{\infty} t^{v-1}\log^{j-r} t \bigg(\frac{1}{e^t-1}-\frac{1}{t}+\frac{1}{2}\bigg)\frac{\partial^r}{\partial v^r}[\mathrm{Li}_{-\frac{v-1}{2}}(e^{-ut})] dt\nonumber\\
		&=-\sum_{r=0}^{k}\sum_{j=r}^{k}\int_{0}^{\infty}t^{v-1}\bigg(\frac{1}{e^t-1}-\frac{1}{t}+\frac{1}{2}\bigg) \binom{j}{r} a_{k+1,j} \log^{j-r}(t) \frac{\partial^r}{\partial v^r}[\mathrm{Li}_{-\frac{v-1}{2}}(e^{-ut})] dt\nonumber\\
		&=-\sum_{r=0}^{k}\sum_{j=0}^{k-r}\int_{0}^{\infty}t^{v-1}\bigg(\frac{1}{e^t-1}-\frac{1}{t}+\frac{1}{2}\bigg) \binom{j+r}{r} a_{k+1,j+r} \log^{j}(t) \frac{\partial^r}{\partial v^r}[\mathrm{Li}_{-\frac{v-1}{2}}(e^{-ut})] dt,
	\end{align}
where in the penultimate step, the order of the double summation has been interchanged and in the last step, we substituted $j$ by $j+r$ in the second summation. Lemma \ref{Dixit's-Lemma} simplifies the following expression as, 
\begin{align}\label{In-terms-of-S(k-r)t}
	\sum_{j=0}^{k-r}\binom{j+r}{r} a_{k+1,j+r} \log^j(t)&=\sum_{j=0}^{k-r}\binom{j+r}{r} \binom{k}{j+r}\lim_{s\to 1}\frac{d^{k-r-j}}{d s^{k-j-r}}\left[\frac{1}{\Gamma(s)}\right] \log^{j}(t)\nonumber\\
	&=\sum_{j=0}^{k-r}\binom{k}{r}\binom{k-r}{j}\lim_{s\to 1}\frac{d^{k-r-j}}{d s^{k-j-r}}\left[\frac{1}{\Gamma(s)}\right] \log^{j}(t)\nonumber\\
	&=\binom{k}{r}\sum_{j=0}^{k-r} a_{k+1-r,j} \log^{j}(t)= \binom{k}{r} S_{k+1-r}(t),
\end{align}
where in the last step, the definition \eqref{Ipoly} of $S_k(t)$ has been applied. Inserting \eqref{In-terms-of-S(k-r)t} into \eqref{Akj-H}, we arrive at
\begin{align*}
		&\sum_{j=0}^{k}a_{k+1,j}\frac{\partial^{(j)}}{\partial v^j}[\tilde{H}(u,v)]\\&=-\sum_{r=0}^{k}\binom{k}{r} \int_{0}^{\infty}t^{v-1}\bigg(\frac{1}{e^t-1}-\frac{1}{t}+\frac{1}{2}\bigg) S_{k+1-r}(t) \frac{\partial^r}{\partial v^r}[\mathrm{Li}_{-\frac{v-1}{2}}(e^{-ut})] dt\\
	&=-\sum_{r=0}^{k}\binom{k}{r} \int_{0}^{\infty}t^{v-1}\bigg(\frac{1}{e^t-1}-\frac{1}{t}+\frac{1}{2}\bigg) S_{k+1-r}(t) \sum_{n=1}^{\infty} n^{\frac{v-1}{2}}\log^r(\sqrt{n}) e^{-nxt} dt.
\end{align*}
Therefore, the integral representation in Lemma \ref{Integral-Representation-New-Fk}, of $\mathfrak{F}_k(u)$  turns the above equation at $v=1$ into
	\begin{align}\label{Sum-H-relation-Fk}
		\sum_{j=0}^{k}a_{k+1,j}	\partial_v^{(j)} \tilde{H}(u,1)=(-1)^{k} \mathfrak{F}_{k+1}(u).		
	\end{align} 
Here, we want to show
	\begin{align*}
		\partial^{(k)}_v\tilde{H}(u,1)
		=\sum_{j=0}^k b_{k,j+1} \mathfrak{F}_k(u).
	\end{align*}
and for that we will apply principle of mathematical induction on $k$. The case $k=0$ directly follows from \eqref{Sum-H-relation-Fk}. Let our claim be true for $0 \le j\le k-1$. Thus, it follows from \eqref{Sum-H-relation-Fk} that
	\begin{align*}
		\partial_v^{(k)} \tilde{H}(u,1)&=-\sum_{j=0}^{k-1}a_{k+1,j}	\sum_{i=0}^j b_{j,i+1} \mathfrak{F}_i(u)-(-1)^k\mathfrak{F}_{k+1}(u)\nonumber\\
		&=-\sum_{i=0}^{k-1}\sum_{j=i}^{k-1} a_{k+1,j}b_{j,i+1} \mathfrak{F}_i(u)-(-1)^k\mathfrak{F}_{k+1}(u),
	\end{align*}
where in the last step we have interchanged the order of the double summation. Finally, the definition \eqref{bkj} of the coefficients $b_{k,i+1}$ concludes that the induction hypothesis is true for $j=k$. This completes the proof of the lemma.
\end{proof}

We are now ready to evaluate the Laurent coefficients of the function $\mathcal{Z}(s,w,w^\prime)$ at $s=1$.
\subsection{Proof of Theorem \ref{Thm-1.2}} 
We begin by invoking the definition \eqref{zetatilde} of $\tilde{\z}(s,(a_1,a_2),(x_1,x_2))$ and then apply Proposition \ref{Prop2.2} to express  the function $\mathcal{Z}(s,w,w^\prime)$ in the form
	\begin{align}\label{Z decomp}
		\mathcal{Z}(s,w,w^\prime)&=(w-w^\prime)^s\tilde{\z}(s,(w^\prime,w),(0,1))\nonumber\\
	&=A(s)\bigg[I_1(s)+I_2(s)+I_3(s)\bigg],
	\end{align}
for $\Re(s)>2$, where $A(s)$ denotes the function 
\begin{align*}
A(s):= \frac{(w-w^{\prime})^{1-s}}{\Gamma^2(s)},
\end{align*}
and the integrals $I_1(s)$, $I_2(s)$ and $I_3(s)$ are already defined in Proposition \ref{Prop2.2}. We next determine the Laurent series of the functions $A(s)$, $I_1(s)$, $I_2(s)$ and $I_3(s)$ respectively at $s=1$. 

The function $A(s)$ is analytic at $s=1$ and for $i\geq 0$, the $i$-th coefficient of the Laurent series expansion at $s=1$ is given by (cf. \cite[p. 66]{Ish03})	
	\begin{align}\label{A(s)Coeff}
		a_i =\sum_{l+m=i}\frac{(-1)^l}{l!} \sum_{\sum_{r=1}^mrs_r=m} \bigg[\prod_{j=1}^m(s_j+1)\bigg(\frac{-\Gamma^{(j)}(1)}{j!}\bigg)^{s_j}\bigg]\log^l(w-w^\prime).
	\end{align}

The $j$-th Laurent coefficient $b_j$ of the integral  $I_1(s)$ at $s=1$ can be determined from $I_1^{(j)}(1)$, the $j$-th derivative of $I_1(s)$ at $s=1$, which is given by
\begin{align*}
I_1^{(j)}(1)&=\sum_{r=0}^j \binom{j}{r} 2^r \int_{w^\prime}^w \log^{j-r}\{(w-u)(u-w^\prime)\} \frac{ \partial}{\partial u}[\partial^r_s\tilde{H}(u,1)] du.
\end{align*}
Invoking Lemma \ref{lem2.3}, the integral appeared in the final term of the above finite sum evaluates as
\begin{align*}
\int_{w^\prime}^w\frac{ \partial\partial^j_s\tilde{H}(u,1)}{\partial u}du = \partial^j_s\tilde{H}(w,1)-\partial^j_s\tilde{H}(w^\prime,1)
= \sum_{l=0}^j b_{j,l+1}\bigg[\mathfrak{F}_{l+1}(w)-\mathfrak{F}_{l+1}(w^{\prime})\bigg].
\end{align*}
Therefore, we evaluate
	\begin{align}\label{I1coeff}
		I_1^{(j)}(1)=&\sum_{r=0}^{j-1} \binom{j}{r} 2^r \int_{w^\prime}^w \log^{j-r}\{(w-u)(u-w^\prime)\} \frac{ \partial}{\partial u}[\partial^r_s\tilde{H}(u,1)] du \nonumber\\
		&\hspace{7cm}+ 2^j\sum_{l=0}^j b_{j,l+1}\bigg[\mathfrak{F}_{l+1}(w)-\mathfrak{F}_{l+1}(w^{\prime})\bigg].
	\end{align}
The other integral $I_2(s)$ is the product of three functions, which are analytic at $s=1$, namely
\begin{align*}
		I_2(s)=\Gamma(2s-1)\z(s-1)g(s)
\end{align*}
where 
\begin{align*}
		g(s)=\int_{w^\prime}^w u^{1-2s}\{(w-u)(u-w^\prime)\}^{s-1} du.
	\end{align*}
The Laurent series of $I_2(s)$ at $s=1$ is essentially the Cauchy product of the Laurent series of $\Gamma(2s-1)$, $\z(s-1)$ and $g(s)$ at $s=1$. We have the Laurent expansions 
	\begin{align*}
		\Gamma(2s-1)=\sum_{n=0}^{\infty}\frac{2^n\Gamma^{(n)}(1)}{n!}(s-1)^n, \quad \z(s-1)=\sum_{n=0}^\infty \frac{\z^{(n)}(0)}{n!}(s-1)^n
	\end{align*}
 at $s=1$. Also, the $n$-th derivative of $g(s)$ at $s=1$ can be simplified to
\begin{align*}
		g^{(n)}(1) = \sum_{q+r\le n}(-2)^r \binom{n}{r} \binom{n-r}{q} I_{r,q,n-r-q}(w,w^\prime),
	\end{align*}
where the integral $I_{a,b,c}(u,v)$ is already defined in \eqref{int I}. Thus, for $j\geq 0$, the $j$-th Laurent coefficient $c_j$ of $I_2(s)$ at $s=1$ can be written as
	\begin{align}\label{I2coeff}
		c_j =\sum_{l_1+l_2+l_3=j}\sum_{q+r\le l_3} \frac{(-1)^r 2^{l_2+r}\z^{(l_1)}(0)\Gamma^{(l_2)}(1)}{l_1!l_2!l_3!} \binom{l_3}{r}\binom{l_3-r}{q} I_{r,q,l_3-r-q}(w,w^\prime).
	\end{align}
The  third integral $I_3(s)$ is also the product of three functions namely,
\begin{align*}
		I_3(s)=\frac{1}{2}\Gamma(2s)\z(s) h(s),
\end{align*}
where
\begin{align*}
		h(s)=\int_{w^\prime}^w u^{-2s}\{(w-u)(u-w^\prime)\}^{s-1} du.
	\end{align*}
Therefore, the Laurent series of $I_3(s)$ at $s=1$ is essentially the Cauchy product of the Laurent series of $\Gamma(2s)$, $\z(s)$ and $h(s)$ at $s=1$. Here at $s=1$, the function $\Gamma(2s)$ is analytic  and the function $\z(s)$ has a simple pole. The Laurent expansions of $\Gamma(2s)$ and $\z(s)$ at $s=1$ are given by
	\begin{align*}
	\Gamma(2s)=\sum_{n=0}^{\infty}\frac{2^n\Gamma^{(n)}(2)}{n!}(s-1)^n, \quad \z(s)=\sum_{n=-1}^{\infty} \tilde{\gamma_n}(s-1)^n.
	\end{align*}
The integral $h(s)$ is analytic at $s=1$ and by applying an argument similar to the one used to derive the Laurent series of $g(s)$ at $s=1$, we obtain the expansion
	\begin{align*}
		h(s)= \sum_{n=0}^{\infty} \left[\sum_{q+r\le n}(-2)^r \binom{n}{r} \binom{n-r}{q} J_{r,q,n-r-q}(w,w^\prime)\right] (s-1)^n,
	\end{align*}
where the integral $J_{a,b,c}(u,v)$ is already defined in \eqref{int J}.	Thus, for $j\geq -1$, the $j$-th Laurent coefficient $d_j$ of $I_3(s)$ at $s=1$ can be written as
	\begin{align}\label{I3coeff}
		d_j =\frac{1}{2}\sum_{\substack{l_1+l_2+l_3=j \\ l_1\ge -1}}\sum_{q+r\le l_3}\frac{(-1)^r2^{r+l_2} \tilde{\gamma}_{l_1}\Gamma^{(l_2)}(2)}{l_2! l_3!}\binom{l_3}{r} \binom{l_3-r}{q} J_{r,q,l_3-r-q}(w,w^\prime).
	\end{align} 
Finally by applying \eqref{A(s)Coeff}, \eqref{I1coeff}, \eqref{I2coeff} and \eqref{I3coeff}, we have the Laurent series of $A(s), I_1(s), I_2(s)$ and $I_3(s)$ at $s=1$, which can be inserted into the decomposition \eqref{Z decomp} to obtain the Laurent series of $\mathcal{Z}(s,w,w^\prime)$ at $s=1$. This completes the proof of the theorem.	 \qed

We next explicitly evaluate the principal part and the constant term in the Laurent expansion of $\mathcal{Z}(s,w,w^\prime)$ at $s=1$ from Theorem \ref{Thm-1.2}.
\subsection{Proof of Corollary \ref{cor Rahul}}
 It follows from Theorem \ref{Thm-1.2} that the coefficient of the principal part in the Laurent expansion of  $\mathcal{Z}(s,w,w^\prime)$ at $s=1$ is given by
    	 \begin{align*}
    		\tilde{P}_{-1}(w,w^\prime)=a_0d_{-1}.
    	\end{align*}
where $a_0=1$ and $d_{-1} = \frac{1}{2}J_{0,0,0}(w,w^\prime)$. The definition \eqref{int J} reduces $d_{-1}$ as
    	\begin{align}\label{d-1}
    		d_{-1} = \frac{1}{2}\int_{w^\prime}^{w} \frac{1}{x^2} dx=\frac{w-w^\prime}{2ww^\prime}.
    	\end{align}
Therefore, the coefficient $\tilde{P}_{-1}(w,w^\prime)$ can be evaluated as  
\begin{align*}
\tilde{P}_{-1}(w,w^\prime) = \frac{w-w^\prime}{2ww^\prime}.
\end{align*}   	
Employing Theorem \ref{Thm-1.2}, we can write the constant term in the Laurent coefficient of  $\mathcal{Z}(s,w,w^\prime)$ at $s=1$ as
    	\begin{align}\label{P0}
    		\tilde{P}_0(w,w^\prime)= a_0(b_0+c_0+d_0)+a_1d_{-1},
    	\end{align}
    	where \begin{align*}
    		&a_1= -\log{(w-w^\prime)}-2\gamma, \, 
    		b_0=\mathfrak{F}_1(w)-\mathfrak{F}_1(w^\prime),\, 
    		c_0=-\frac{1}{2}\log{\frac{w}{w^\prime}},\\
    		&d_0=\frac{\gamma}{2}\frac{w-w^\prime}{ww^\prime}+(1-\gamma)\frac{w-w^\prime}{ww^\prime}+	\frac{1}{2}J_{0,0,1}(w,w^\prime)+\frac{1}{2}J_{0,1,0}(w,w^\prime)-J_{1,0,0}(w,w^\prime).
    	\end{align*}

We next determine the value of the constant $d_0$ by evaluating the integrals  $J_{0,0,1}(w,w^\prime)$, $J_{0,1,0}(w,w^\prime)$ and $J_{1,0,0}(w,w^\prime)$ respectively from definition  \eqref{int J}. The improper integral $J_{0,0,1}(w,w^\prime)$ can be derived as
\begin{align*}
    		J_{0,0,1}(w,w^\prime)&=\int_{w^\prime}^{w} \frac{\log{(x-w^\prime)}}{x^2} dx \\
    		&= \lim_{\epsilon \to 0^+} \int_{w^\prime+\epsilon}^{w} \frac{\log{(x-w^\prime)}}{x^2} dx 
    		=\frac{w-w^\prime}{ww^\prime}\log{(w-w^\prime)} -\frac{\log{w}}{w^\prime}+\frac{\log{w^\prime}}{w^\prime},
    	\end{align*}
where we mainly applied integration by parts and the fact $\lim\limits_{\epsilon \to 0^+}\epsilon \log{\epsilon}=0$. Following the similar argument, the integral $J_{0,0,1}(w,w^\prime)$ evaluates as
\begin{align*}
    		J_{0,1,0}(w,w^\prime) = \frac{w-w^\prime}{ww^\prime}\log{(w-w^\prime)} -\frac{\log{w}}{w}+\frac{\log{w^\prime}}{w}.
    	\end{align*}
Applying integration by parts, the third integral $J_{1,0,0}(w,w^\prime)$ simplifies to
    		\begin{align*}
    		J_{1,0,0}(w,w^\prime)= \int_{w^\prime}^{w} \frac{\log x}{x^2} dx=-\frac{w-w^\prime}{ww^\prime}-\frac{\log{w}}{w}+\frac{\log{w^\prime}}{w^\prime}.
    	\end{align*}
Thus, the value of the above three integrals determines the constant $d_0$ as
    		\begin{align}\label{d0}
    		d_0 =\frac{w-w^\prime}{ww^\prime}\log{(w-w^\prime)}-\gamma\frac{w-w^\prime}{2ww^\prime}-\frac{w-w^\prime}{2ww^\prime}\log{(ww')}.
    	\end{align}
Finally, by inserting the value of the constants $a_0, a_1, b_0, c_0, d_{-1}$ and $d_0$ from \eqref{d-1} and \eqref{d0} into \eqref{P0}, the constant term in the Laurent expansion of $\mathcal{Z}(s,w,w^\prime)$ at $s=1$, can be derived as 
    	\begin{align*}
    		\tilde{P}_0(w,w^\prime)= a_0(b_0+c_0+d_0)+a_1d_{-1}=\mathfrak{F}_1(w)-\mathfrak{F}_1(w')-\frac{1}{2}\log{\frac{w}{w'}}+\frac{w-w^\prime}{2ww^\prime}\bigg[\gamma+\log{\frac{w-w'}{ww'}}\bigg].
    	\end{align*} 
This concludes the proof of the corollary.   	\qed
 
 \section{Arithmetic properties of $\mathfrak{F}_k^0(x)$}\label{Arithmetic properties}
The function $\mathfrak{F}_k(x)$ plays a central role in the Laurent series expansion of $\mathcal{Z}(s,w,w^\prime)$. This section concerns about the arithmetic properties of the function $\mathfrak{F}_k^0(x)$, defined in \eqref{Dixit's-Fk}, which is basically the first term of the finite sum in the function $\mathfrak{F}_k(x)$ and for that we have investigated the arithmetic properties of the modified function $\mathcal{F}_k(x)$ defined in \eqref{Generalized-Ramanujan-Period-Function}. The following lemma states the integral representation of the function $\mathfrak{F}_k^0(x)$.
\begin{lemma}\label{Lemma:3.2}
For $\Re(x)>0$, we have
     \begin{align*}
        \mathfrak{F}_k^0(x)=(-1)^{k}\int_{0}^{\infty}\bigg[\frac{1}{e^{t}-1}-\frac{1}{t}+\frac{1}{2}\bigg] \frac{S_k(t)}{e^{xt}-1} dt.
    \end{align*}
\end{lemma}
\noindent The proof follows along the similar lines of the proof of Lemma \ref{Integral-Representation-New-Fk}. We next consider the integral 
\begin{align}\label{2.4}
	\mathcal{I}_k(x,s):=- \int_0^\infty \Bigg[ \frac{1}{(e^t-1)(e^{xt}-1)}+ \frac{1}{2(e^{xt}-1)}+\frac{1}{2(e^t-1)}\Bigg] \frac{d^{k-1}}{ds^{k-1}}\Bigg[ \frac{t^{s-1}}{\Gamma(s)}\Bigg] dt,
\end{align}
which is defined for $\Re(x)>0$ and $\Re(s)>0$. In the next lemma, we establish that as $s \to 1$ the above integral $\mathcal{I}_k(x,s)$ reduces to $\mathcal{F}_k(x)$. 
\begin{lemma}\label{modiFkx}
For $\Re(x)>0$, we have  $\lim\limits_{s \to1 }	\mathcal{I}_k(x,s)=\mathcal{F}_k(x)$.
\end{lemma}
\begin{proof}
We begin by decomposing the integral $\mathcal{I}_k(x,s)$ as
\begin{align}\label{modFkexp}
	\mathcal{I}_k(x,s)&= -\int_{0}^{\infty} \bigg( \frac{1}{e^t -1}+ \frac{1}{2}-\frac{1}{t}\bigg) \frac{1}{e^{xt}-1} \frac{d^{k-1}}{ds^{k-1}}\Bigg[ \frac{t^{s-1}}{\Gamma(s)}\Bigg] dt - \int_{0}^{\infty} \frac{1}{t(e^{xt}-1)} \frac{d^{k-1}}{ds^{k-1}}\Bigg[ \frac{t^{s-1}}{\Gamma(s)}\Bigg] dt \nonumber\\
	& \hspace{9.2cm} - \frac{1}{2} \int_{0}^{\infty} \frac{1}{e^t-1} \frac{d^{k-1}}{ds^{k-1}}\Bigg[ \frac{t^{s-1}}{\Gamma(s)}\Bigg] dt \nonumber\\
&= -\int_{0}^{\infty} \bigg( \frac{1}{e^t -1}+ \frac{1}{2}-\frac{1}{t}\bigg) \frac{1}{e^{xt}-1} \frac{d^{k-1}}{ds^{k-1}}\Bigg[ \frac{t^{s-1}}{\Gamma(s)}\Bigg]  dt  - \frac{d^{k-1}}{ds^{k-1}}\Bigg[\frac{x^{(1-s)}\zeta(s-1)}{s-1}\Bigg]\nonumber\\
&\hspace{10.5cm} -\frac{1}{2}\frac{d^{k-1}}{ds^{k-1}} \Bigg[ \zeta(s)\Bigg],
\end{align}
where the last step follows from the well-known integral representation  (cf. \cite[p. 251, Theorem 12.2]{Apostol})
\begin{align}\label{Gammazetaintegral}
\Gamma(s)\zeta(s)=\int_{0}^{\infty}\frac{t^{s-1}}{e^t-1} \ dt \quad \qquad (\Re(s)>1).
\end{align}
Letting $s \to 1$ on the both side of \eqref{modFkexp}, we invoke Lemma \ref{Dixit's-Lemma}, Lemma \ref{Lemma:3.2} on the first term and apply the Laurent series expansions at $s=1$ of the second and third term in right hand side of  \eqref{modFkexp} to conclude that
\begin{align*}
	 \lim_{s\to 1}\mathcal{I}_k(x,s)&=(-1)^{k-1}\mathfrak{F}_k^0(x)-\sum_{l=0}^{k}\frac{(-1)^l}{k}\binom{k}{l} \log^{l}(x)\zeta^{(k-l)}(0)-\frac{1}{2}(-1)^{k-1}\gamma_{k-1} =	\mathcal{F}_k(x). 
\end{align*}
\end{proof}
Lewis and Zagier \cite[p. 228]{Lewis} introduced variety of functions satisfying three term functional equation. Among them, one of the non-trivial example was the double sum
 \begin{align}\label{psi+}
	\psi^+_s(x)= {\sum_{m,n\ge 0}}{}^{{}^{\hspace{-.1cm}*}} \  \frac{1}{(mx+n)^{2s}},
\end{align}
which was defined for $\Re(s)>1$ and $x \in \mathbb{C} \setminus (-\infty,0]$. Here ${\sum}{}^{{}^{*}}$  means the term $m=n=0$ is to be omitted and the terms with either $m$ or $n$ equal to zero are to be counted with multiplicity $\frac{1}{2}$. Choie and Kumar \cite[p. 12]{Choie} extensively studied the analytic properties of the above function. As $s \to 1/2$, the function  $\psi^+_s(x)$ reduces to Ramanujan Period function $\mathcal{F}_1(x)$, which led the authors to term the function $\psi_s^+(x)$ as generalized Ramanujan period function. 

Now, for simplicity, we abbreviate the function $\psi_{\frac{s}{2}}^+(x)$ by $\Psi_{s}(x)$ for $\Re(s)>2$.
The following lemma produces a close relationship between the function $\Psi_s(x)$ and the integral $\mathcal{I}_k(x,s)$.
\begin{lemma}\label{Ik(s)-derivative-of-psi(x)}
For $\Re(s)>2$ and $\Re(x)>0$, we have
	\begin{align*}
		\mathcal{I}_k(x,s)=-\frac{d^{k-1}}{ds^{k-1}} \bigg[ \Psi_s(x) \bigg].
	\end{align*}
\end{lemma}
\begin{proof}
The integral representation of $\Psi_s(x)$ is given by \cite[Proposition 6.1]{Choie}
\begin{align}\label{psi-integral}
		\Psi_s(x)=\int_{0}^{\infty} \left(\frac{1}{e^t-1}+\frac{1}{2}+\frac{x^s}{2}\right)\frac{\frac{t^{s-1}}{\Gamma(s)}}{e^{xt}-1} dt,
	\end{align}	
which is valid for $\Re(s)>2$ and $\Re(x)>0$. We next split the above  integral into two parts and then substitute $t$ by $t/x$ in the second integral, to reduce the integral \eqref{psi-integral} as
	 \begin{align}
	 		\Psi_s(x)&=\int_{0}^{\infty} \left(\frac{1}{e^t-1}+\frac{1}{2}\right)\frac{\frac{t^{s-1}}{\Gamma(s)}}{e^{xt}-1} dt +\frac{x^s}{2}\int_{0}^{\infty} \frac{\frac{t^{s-1}}{\Gamma(s)}}{e^{xt}-1} dt \nonumber\\
	 &=\int_{0}^{\infty} \bigg[ \frac{1}{(e^t -1)(e^{xt}-1)}+\frac{1}{2(e^{xt}-1)}+ \frac{1}{2(e^t -1)}\bigg] \frac{t^{s-1}}{\Gamma(s)} dt.\nonumber		
	 \end{align}
Finally, by differentiating both sides $k-1$ times and applying the definition \eqref{modiFkx} of  $\mathcal{I}_k(x,s)$, one can conclude our lemma.	 
\end{proof}

We are now ready to prove the two term and the three term functional equation for $\mathcal{F}_k(x)$.
\subsection{\textbf{Proof of Theorem \ref{Periodlike}} }

The series definition \eqref{psi+} of $\psi^+_s(x)$ readily yields the  functional equation
	\begin{align*}
		\Psi_s(x)= x^{-s}\Psi_s\(\frac{1}{x}\).
	\end{align*}
Upon differentiating both sides $k-1$ times with respect to $s$, an application of Lemma \ref{Ik(s)-derivative-of-psi(x)} and general Leibnitz rule for differentiation provides
\begin{align*}
	\mathcal{I}_k(x, s)= \frac{1}{x^s} \sum_{j=1}^k \binom{k-1}{j-1} \log^{k-j}\(\frac{1}{x}\) \mathcal{I}_{j}\(\frac{1}{x}, s\).
\end{align*}
Therefore, by taking the limit as $s\to 1$ on both sides, Lemma \ref{modiFkx} reduces the above equation to \eqref{two term}. 

Lewis and Zagier in \cite{Lewis} established that the function $\Psi_s(x)$ also satisfies the three term relation:
\begin{align*}
	\Psi_s(x)-\Psi_s(x+1)-\frac{1}{(x+1)^{s}}\Psi_s(\frac{x}{x+1})=0.
\end{align*}
Performing $k-1$ times differentiation  with respect to $s$, we invoke Lemma \ref{Ik(s)-derivative-of-psi(x)} to reduce the above equation into
\begin{align*}
\mathcal{I}_k(x+1,s)=\mathcal{I}_k(x,s)-\frac{1}{(x+1)^s} \sum_{j=1}^{k}\binom{k-1}{j-1}  \log^{k-j}\(\frac{1}{x+1}\) \mathcal{I}_j\(\frac{x}{x+1},s\).
\end{align*}
Finally, by letting $s \to 1$ on the both sides, Lemma \ref{modiFkx} concludes \eqref{three term}. This completes the proof of the theorem. \qed

We next focus on the action of Hecke operator $\tilde{T_n}$ on the partial sum of $\mathcal{F}_k(x)$.

\subsection{\textbf{Proof of Theorem \ref{Heckeopaction}}}
We have from \cite[Theorem 6.2]{Choie} that the function $ \Psi_s(x)$ is the Hecke eigenform with respect to the Hecke operator $\tilde{T_n}$:
\begin{align*}
     (\Psi_s|_{s}\tilde{T}_n)(x)= n^{\frac{s}{2}}\sigma_{1-s}(n)\Psi_s(x).
\end{align*}
Therefore, by applying the definition \eqref{genslash} of the slash operator and the Hecke operator $\tilde{T_n}$, we can rewrite the above equation as
    \begin{align*}
       \sum_{\g}v_{\g}\frac{1}{(cx+d)^s}\Psi_s(\frac{ax+b}{cx+d})=\sum_{l|n}l^{1-s}\Psi_s(x).
    \end{align*}
We next perform $k-1$ times differentiation with respect to $s$ on both sides of the above equation to obtain
    \begin{align*}
        \sum_{\g}v_{\g}\sum_{j=1}^k\binom{k-1}{j-1}\frac{1}{(cx+d)^s}&\log^{k-j}\left(\frac{1}{cx+d}\right) \frac{d^{j-1}}{ds^{j-1}}\bigg[\Psi_s\left(\frac{ax+b}{cx+d}\right)\bigg]\\
       &= \sum_{j=1}^k\binom{k-1}{j-1}\sum_{l|n}l^{1-s}\log^{k-j}\left(\frac{1}{l}\right)\frac{d^{j-1}}{ds^{j-1}}\bigg[\Psi_s(x)\bigg].
    \end{align*}
Lemma \ref{Ik(s)-derivative-of-psi(x)} reduces the above equation into
\begin{align*}
\sum_{\g}v_{\g}\sum_{j=1}^k\binom{k-1}{j-1}\frac{1}{(cx+d)^s}\log^{k-j}\hspace{-.15cm}\left(\frac{1}{cx+d}\right) \mathcal{I}_j(x,s)
       = \sum_{j=1}^k\hspace{-.15cm}\binom{k-1}{j-1}\hspace{-.15cm}\sum_{l|n}l^{1-s}\log^{k-j}\hspace{-.15cm}\left(\frac{1}{l}\right)\hspace{-.15cm}\mathcal{I}_j(x,s).
\end{align*}    
Letting $s \to 1$ on the both sides and invoking Lemma \ref{modiFkx}, we arrive at 
\begin{align*}
        \sum_{\g}v_{\g}\sum_{j=1}^k\binom{k-1}{j-1}\frac{\sqrt{n}}{(cx+d)}\log^{k-j}\left(\frac{1}{cx+d}\right) &\mathcal{F}_j\left(\frac{ax+b}{cx+d}\right)\\
        &=\sqrt{n} \sum_{j=1}^k\binom{k-1}{j-1}\sum_{l|n}\log^{k-j}\left(\frac{1}{l}\right)\mathcal{F}_j(x).
 \end{align*}
Finally, we apply the definition \eqref{genslash} of the general slash operator and the definition of the Hecke operator $\tilde{T_n}$ to conclude our theorem.  \qed
\section{Application of $\mathfrak{F}_k^0(x)$ in the theory of integration}\label{Application}
In this section, we establish a relation between the function  $\mathfrak{F}_k^0(x)$ and an integral  $J_k(x) $, defined in \eqref{J_k-Function}. A straight forward application of Lemma \ref{Dixit's-Lemma} has been shown in the following lemma, which express $S_k\left(\frac{t}{2}\right)$ in terms of $S_k(t)$, defined in \eqref{Ipoly}.
\begin{lemma}\label{sktrelation}
	For any positive integer $k$, we have
	\begin{align*}
		S_k\(\frac{t}{2}\)=\sum_{j=1}^{k}\binom{k-1}{j-1}\log^{k-j}\(\frac{1}{2}\) S_j(t).
	\end{align*}
\end{lemma}
\begin{proof}
Lemma \ref{Dixit's-Lemma} yields
	\begin{align*}
			S_k\left(\frac{t}{2}\right)= \lim_{s\to1} \frac{d^{k-1}}{ds^{k-1}}\bigg[ \frac{(t/2)^{s-1}}{\Gamma(s)}\bigg].
	\end{align*}
Therefore, by applying the general Leibnitz rule of differentiation on the right hand side of the above equation, we obtain 
	\begin{align*}
			S_k\left(\frac{t}{2}\right)&= \lim_{s\to1}\sum_{j=1}^{k}\binom{k-1}{j-1}\frac{1}{2^{s-1}}(-1)^{k-j}\log^{k-j}{(2)}\frac{d^{j-1}}{ds^{j-1}}\bigg[ \frac{t^{s-1}}{\Gamma(s)}\bigg]\\
			&=\sum_{j=1}^{k}\binom{k-1}{j-1}(-1)^{k-j}\log^{k-j}{(2)} S_j(t),
	\end{align*}
which concludes our lemma.
\end{proof}

We next express the integral $J_k(x)$ in terms of $\mathfrak{F}_k^0(x)$.
\subsection{\textbf{Proof of Theorem \ref{JrelF}}}
\begin{proof}
We first split the integral $J_k(x)$ as sum of three integrals namely,
\begin{align}\label{Jsplit}
	J_k(x) = I_k^1(x)+I_k^2(x)+I_k^3(x),
\end{align}
where
\begin{align}
   & I_k^1(x)=-\int_{0}^\infty \bigg[\frac{1}{1-e^{-t}}-\frac{1}{t}-\frac{1}{2} \bigg]\frac{S_k(t)}{1+e^{xt}}~dt,\nonumber\\
   & I_k^2(x)=\int_{0}^{\infty}\bigg[\frac{2}{1-e^{-2t}}-\frac{1}{t}-1\bigg]\frac{S_k(t)}{1+e^{xt}}~dt\qquad \text{and} \quad I_k^3(x)=\frac{1}{2}\int_{0}^\infty\frac{S_k(t)}{1+e^{xt}}dt.\nonumber
\end{align}
The integral representation of $\mathfrak{F}_k^0(x)$ from Lemma \ref{Lemma:3.2}, derives the integral $I_k^1(x)$ as
\begin{align}\label{Ik1-relation}
    I_k^1(x)
    &=-\int_{0}^\infty \bigg[\frac{1}{1-e^{-t}}-\frac{1}{t}-\frac{1}{2} \bigg] \frac{e^{xt}+1-2}{e^{2xt}-1} S_k(t) \ dt\nonumber\\
    &= 2 \int_{0}^\infty \bigg[\frac{1}{e^{t}-1}-\frac{1}{t}+\frac{1}{2} \bigg]\frac{S_k(t)}{e^{2xt}-1} \ dt - \int_{0}^\infty \bigg[\frac{1}{e^{t}-1}-\frac{1}{t}+\frac{1}{2} \bigg]\frac{S_k(t)}{e^{xt}-1} \ dt \nonumber\\
    &=(-1)^k\bigg[2\mathfrak{F}_k^0(2x)-\mathfrak{F}_k^0(x)\bigg].
\end{align}
Similarly, the integral $I_k^2(x)$ can be simplified to
\begin{align}\label{Ik2-relation}
	I_k^2(x)&= 2 \int_{0}^\infty \bigg[\frac{1}{e^{2t}-1}-\frac{1}{2t}+\frac{1}{2} \bigg]\frac{e^{xt}+1-2}{e^{2xt}-1} S_k(t) \ dt\nonumber\\
	&= \int_{0}^{\infty}\bigg[\frac{1}{e^{t}-1}-\frac{1}{t}+\frac{1}{2} \bigg]\frac{S_k(t/2)}{e^{xt/2}-1}~dt  - 2 \int_{0}^{\infty}\bigg[\frac{1}{e^{t}-1}-\frac{1}{t}+\frac{1}{2} \bigg]\frac{S_k(t/2)}{e^{xt}-1}~dt\nonumber\\
	&= \sum_{j=1}^{k}\binom{k-1}{j-1}\log^{k-j}\left(\frac{1}{2}\right)\bigg[\int_{0}^{\infty}\bigg[\frac{1}{e^{t}-1}-\frac{1}{t}+\frac{1}{2} \bigg]\frac{S_j(t)}{e^{xt/2}-1}dt \nonumber\\
	&\hspace{7cm}- 2 \int_{0}^{\infty}\bigg[\frac{1}{e^{t}-1}-\frac{1}{t}+\frac{1}{2} \bigg]\frac{S_j(t)}{e^{xt}-1}dt \bigg]\nonumber\\
	&=\sum_{j=1}^k (-1)^j \binom{k-1}{j-1}\log^{k-j}\(\frac{1}{2}\) \bigg[\mathfrak{F}_j^0\(\frac{x}{2}\)-2\mathfrak{F}_j^0(x)\bigg],
\end{align}
where in the second step, we make a change of variable $t$ by $\frac{t}{2}$ and in the penultimate step, we have applied Lemma \ref{sktrelation}. The integral representation \eqref{Gammazetaintegral} of $\Gamma(s)\zeta(s)$ along with Lemma \ref{Dixit's-Lemma} evaluates the integral $I_k^3(x)$ as
\begin{align}\label{Ik3-relation}
    I_k^3(x)
    &=\frac{1}{2}\int_{0}^\infty \frac{e^{xt}+1-2}{e^{2xt}-1} S_k(t) dt\nonumber\\
    &=\frac{1}{2}\lim_{s\to 1}\frac{d^{k-1}}{ds^{k-1}}\bigg[x^{-s}\zeta(s)-2^{1-s}x^{-s}\zeta(s)\bigg]\nonumber\\
    &=\frac{(-1)^k}{2kx}\bigg[\log^k x-\log^k2x\bigg]-\frac{(-1)^{k}}{2x}\sum_{l=0}^{k-1}\binom{k-1}{l}\gamma_{k-1-l}\bigg[\log^lx-\log^l2x\bigg],
\end{align}
where the final step follows from $(k-1)$ times differentiation of the Laurent series expansions of the functions appeared in the penultimate step and then evaluation at $s=1$. Finally, inserting \eqref{Ik1-relation}, \eqref{Ik2-relation} and \eqref{Ik3-relation} into \eqref{Jsplit}, we arrive at our result.
 \end{proof}
 
\subsection*{Acknowledgements} The first author’s research was partially supported by Anusandhan National Research
Foundation (ANRF) grant ANRF/ARGM/2025/000175/MTR of Govt. of India and the Cumulative Professional Development Allowance (CPDA) grant from the institute he is affiliated with. The second author is currently a Ph.D student at IIT Kharagpur and her research was supported by University Grants Commision (UGC), Govt. of India.

\subsection*{Data availability} Data sharing is not applicable to this article as no datasets were generated or analyzed during the current study.

\end{document}